\pdfoutput=1
\pdfoutput=1
\pdfoutput=1
\pdfoutput=1
\pdfoutput=1
\documentclass[a4paper,12pt]{article}
\usepackage{calc}
\usepackage[all]{xy}
\usepackage[centertags]{amsmath}
\usepackage{latexsym}
\usepackage{amsfonts}
\usepackage{graphicx}
\usepackage{tikz}
\usepackage{cases}
\usepackage{amssymb}
\usepackage{amsthm}
\usepackage{color}
\usepackage{fancyhdr}
\usepackage[dvips]{epsfig}
\usepackage{newlfont}
\usepackage[latin1]{inputenc}
\usepackage[latin1]{inputenc}
\usepackage[english,french]{babel}
\usepackage{graphicx}
\usepackage{t1enc}
\usepackage[english,french]{babel}
\usepackage{fancybox}
\usepackage{graphicx}
\usepackage{t1enc}
\usepackage[french2]{minitoc}
\usepackage{mathrsfs}
\usepackage{amsfonts}
\usepackage{wasysym}
\usepackage{hyperref}
\allowdisplaybreaks
\usepackage{float}

\fancypagestyle{plain}{ \fancyhead{}

	\rhead{\textbf{}} \rfoot{\footnotesize{\textsf{\tiny }}}
	\cfoot{\footnotesize{\textbf{ \thepage}}}} \hfuzz2pt
\newlength{\defbaselineskip}
\numberwithin{equation}{section} 
\newtheorem{theorem}{Theorem}[section]
\newtheorem{corollary}[theorem]{Corollary}

\newtheorem{lemma}[theorem]{Lemma}
\newtheorem{proposition}[theorem]{Proposition}

\newtheorem{remark}{Remark}[section]

\newcommand{\Z}{{\mathbb Z}}

 \makeatletter
\newcommand{\thechapterwords}
{ \ifcase \thechapter\or 1\or 2\or 3\or 4\or 5\or
	6\or 7\or 8\or 9\or 10\or 11\fi}
\def\thickhrulefill{\leavevmode \leaders \hrule height 2ex \hfill \kern \z@}
\def\@makechapterhead#1{%
	\vspace*{15\p@}%
	{\parindent \z@ \centering \reset@font
		\thickhrulefill\quad
		\scshape  {\chapnumfont \@chapapp{}}{\chapnumfont \thechapterwords}
		\quad \thickhrulefill
		\par\nobreak
		\vspace*{15\p@}%
		\interlinepenalty\@M
		\hrule
		\vspace*{15\p@}%
		\huge {\bfseries  #1}\par\nobreak
		\par
		\vspace*{15\p@}%
		\hrule
		\vskip 15\p@
}}
\def\@makeschapterhead#1{%
	\vspace*{15\p@}%
	{\parindent \z@ \centering \reset@font
		\thickhrulefill
		\par\nobreak
		\vspace*{15\p@}%
		\interlinepenalty\@M
		\hrule
		\vspace*{15\p@}%
		\Huge \bfseries #1\par\nobreak
		\par
		\vspace*{15\p@}%
		\hrule
		\vskip 30\p@
}}

\DeclareFixedFont{\chapnumfont}{T1}{phv}{b}{n}{20pt}
\DeclareFixedFont{\chapchapfont}{T1}{phv}{b}{n}{16pt}
\DeclareFixedFont{\chaptitfont}{T1}{phv}{b}{n}{24.88pt}
\def\@makechapterhead#1{%
	\vspace*{15\p@}%
	{\parindent \z@ \centering \reset@font
		\thickhrulefill\quad
		\scshape {\chaptitfont\color[rgb]{0.00,0.50,1.00}\@chapapp{}}
		{\chapnumfont \thechapterwords}
		\quad \thickhrulefill
		\par\nobreak
		\vspace*{15\p@}%
		\interlinepenalty\@M
		\hrule
		\vspace*{15\p@}%
		{\Large\bfseries #1}\par\nobreak
		\par
		\vspace*{15\p@}%
		\hrule
		\vskip 30\p@
}}%

\begin{document}
	\title{Infinite Lifting of an  Action of Symplectomorphism Group on the set of  bi-Lagrangian structures}
	\date{ }
	\author{ \large{Bertuel  TANGUE NDAWA }\\
		\normalsize{University Institute of Technology of }\\
		\normalsize{University of Ngaoundere}\\
		\normalsize{Box: 455, Ngaoundere (Cameroon)}\\ \normalsize{bertuelt@yahoo.fr}
		\vspace{0.5cm}\\\today}
	\maketitle
	
	\selectlanguage{english}
	\begin{center}
		To Kira and Thierry Rothen's family.	
	\end{center}
	\section*{Abstract}
	Let $M$ be a smooth $2n$-manifold endowed with a bi-Lagrangian structure $(\omega,\mathcal{F}_{1},\mathcal{F}_{2})$, where $\omega$ is symplectic and $(\mathcal{F}_{1},\mathcal{F}_{2})$ is a pair of transversal Lagrangian foliations. Such structures admit a natural geometric object, the Hess connection, which is central to the classification of affine bi-Lagrangian structures.

	We show that any bi-Lagrangian structure on $M$ lifts to a bi-Lagrangian structure on the trivial bundle $M\times\mathbb{R}^{2n}$. Moreover, the lift of an affine bi-Lagrangian structure is again affine. We then define a dynamical system in terms of an action of the symplectomorphism group on the space of bi-Lagrangian structures. This dynamics is compatible with Hess connections, preserves affine bi-Lagrangian structures, and can be lifted to $M\times\mathbb{R}^{n}$, and iterated to $\left(M\times\mathbb{R}^{2n}\right)\times\mathbb{R}^{4n}$. The resulting lifted dynamics agrees with the initial one on $M\times\mathbb{R}^{n}$ for suitable structures. Finally, the same results hold when $M\times\mathbb{R}^{2n}$ is replaced by $TM$ or $T^{*}M$, provided that $M$ is parallelizable.
	
	\textbf{Keywords}: Symplectic, Symplectomorphism, Bi-Lagrangian, Para-K\"{a}hler, Hess connection.
	
	\textbf{MSC2010}: 53D05, 53D12.
	
	\textbf{Acknowledgment}:
	
	The author would like to thank the members of the Geometry and High Energy Physics research group at the University of Douala for their participation in the presentations on this subject. In particular, the author is grateful to Prof. NGAKEU Ferdinand, who suggested this topic as a master's thesis.

	\section{Introduction}
	Let $(M,\omega)$ be a symplectic manifold. Thus $\omega$ is a symplectic form on $M$, i.e.\ a closed and nondegenerate $2$-form, meaning that $d\omega=0$ and $\omega$ is nondegenerate as a bilinear form on the space of vector fields $\mathfrak{X}(M)$; see \cite{dasilva,paul}.
	
	A bi-Lagrangian structure on $(M,\omega)$ is a pair $(\mathcal{F}_{1},\mathcal{F}_{2})$ of transversal Lagrangian foliations. Equivalently, a bi-Lagrangian structure on $M$ is a triple $(\omega,\mathcal{F}_{1},\mathcal{F}_{2})$, where $(\mathcal{F}_{1},\mathcal{F}_{2})$ is a pair of transversal Lagrangian foliations on the symplectic manifold $(M,\omega)$; see \cite{2,11,3,FR1,FR2,1,7,FE}. In both cases, the quadruple $(M,\omega,\mathcal{F}_{1},\mathcal{F}_{2})$ is called a \emph{bi-Lagrangian manifold}. We refer to \S\ref{sub1} for further details on Lagrangian foliations.

	Let $(M,\omega,\mathcal{F}_{1},\mathcal{F}_{2})$ be a bi-Lagrangian manifold. The \emph{Hess connection} associated with $(M,\omega,\mathcal{F}_{1},\mathcal{F}_{2})$ is the symplectic connection $\nabla$, i.e.\ a torsion-free connection that parallelizes $\omega$ and preserves the two foliations; see \cite{2,11,3,FR1,FR2,1,7,FE}. The existence and uniqueness of such a connection were established in \cite{7}, and its relevance was emphasized in \cite{2,11,3}. Hess connections are special cases of Bott connections, i.e. linear connections preserving the foliations (see \cite{vai,Wei}). Bott connections play an important role in the theory of geometric quantization of real polarizations (see, for instance, \cite{EF}).
	
	It is also worth noting that a bi-Lagrangian structure $(\omega,\mathcal{F}_{1},\mathcal{F}_{2})$ on $M$ corresponds \emph{bijectively} to a para-K\"ahler structure $(G,F)$ on $M$, where $G$ is a pseudo-Riemannian metric and $F$ is a para-complex structure on $M$ satisfying $	G(F(\cdot),F(\cdot))=-G(\cdot,\cdot)$.	The tensors $\omega$, $G$, and $F$ are related by
$\omega(\cdot,\cdot)=G(F(\cdot),\cdot),$
	see \cite{FR1,FR2,1,FE}. Moreover, the Levi-Civita connection of $G$ coincides with the Hess connection of the bi-Lagrangian manifold $(M,\omega,\mathcal{F}_{1},\mathcal{F}_{2})$; see again \cite{FR1,FR2,1,FE}. Consequently, bi-Lagrangian manifolds sit at the intersection of symplectic geometry, semi-Riemannian geometry, and almost product (para-complex) geometry. They appear naturally in geometric quantization (see \cite{7}) and in the study of Koszul--Vinberg cohomology (see \cite{GB1}).

	The main contributions of this paper are as follows. We show that any bi-Lagrangian structure on $M$ can be lifted to a bi-Lagrangian structure on the trivial bundle $M\times\mathbb{R}^{2n}$. Furthermore, the lift of an \emph{affine} bi-Lagrangian structure is again affine. We then introduce a dynamical system in terms of an action of the symplectomorphism group on the space of bi-Lagrangian structures. This dynamics is compatible with Hess connections, preserves affine bi-Lagrangian structures, and admits a further lifting to $M\times\mathbb{R}^{n}$; iterating once more yields a lift to $\bigl(M\times\mathbb{R}^{2n}\bigr)\times\mathbb{R}^{4n}$. In particular, the resulting lifted dynamics agrees with the initial one on $M\times\mathbb{R}^{n}$ for suitable choices of structures. Moreover, the preceding results remain true after replacing $M\times\mathbb{R}^{2n}$ by $TM$ or by $T^{*}M$, whenever $M$ is parallelizable.

	Before stating and proving our results in detail, we need to introduce some definitions, fix notations, and recall several known facts that will be used throughout the paper.
	\subsection{Definitions and notations}\label{sub1}
	We assume that all the objects are smooth throughout this paper.
	
	Let $M$ be an $m$-manifold. By a $p$-dimensional, class $C^r$, $0\leq r\leq \infty$ foliation $\mathcal{F}$ of $M$ we mean a decomposition of $M$ into a union of disjoint connected subsets $\{\mathcal{F}_x\}_{x\in M}$, called the leaves of the foliation, with the following property: every point $y$ in $M$ has a neighborhood $U$ and a system of local, class $C^r$ coordinates  $(y^1,\dots, y^m):U\longrightarrow\mathbb{R}^m$ such that for each leaf $\mathcal{F}_x$ the components of $U\cap\mathcal{F}_x$  are described by the equations $y^{p+1}\mbox{=constant},\dots, y^m=\mbox{constant}$, see \cite{law}.
	
	The expressions $T\mathcal{F}\subset TM$ and $\Gamma\left(T\mathcal{F}\right)$ (or $\Gamma\left(\mathcal{F}\right)$)  denote the tangent bundle to $\mathcal{F}$ and the set of sections of $T\mathcal{F}$  respectively.
	
	Let $\psi : M\longrightarrow N$ be a diffeomorphism. The  push forward  $\psi_*\mathcal{F}=\{\psi_*\mathcal{F}_x\}_{x\in M}$  of $\mathcal{F}$ by $\psi$ is a foliation, and
	\begin{equation}
	\Gamma\left(\psi_*\mathcal{F}\right):=\{\psi_*X,\; X \in\Gamma\left(\mathcal{F}\right)\}=\psi_*\Gamma\left(\mathcal{F}\right).\label{Bieq2}
	\end{equation}
	
	For every $k$-manifold $M'$, the  set $M'\times\mathbb{R}^k$ is called the trivial bundle of $M'$.
	We say that a manifold  is parallelizable when its tangent bundle is diffeomorphic to its trivial bundle. We denote by $\mathcal{M}^{\pi}$ the set of parallelizable manifolds.  Note that:   every Lie group belongs to $\mathcal{M}^{\pi}$; if $M'$ is a $k$-manifold
	which can be covered by a single  chart, then  $M'\in \mathcal{M}^{\pi}$; every  connected 1-manifold is an element of  $\mathcal{M}^{\pi}$; since  the tangent bundle of the product of two manifolds and the product of  their tangent bundles  are diffeomorphic, then  the product of two manifolds in $\mathcal{M}^{\pi}$ also belongs to $\mathcal{M}^{\pi}$.
	
	If the manifold $M$ is endowed with a symplectic form $\omega$ (as a consequence, $m=2n$), a foliation $\mathcal{F}$ is Lagrangian if for every $X\in\Gamma\left(\mathcal{F}\right)$, $\omega(X,Y)=0$ if and only if $Y\in\Gamma\left(\mathcal{F}\right).$
	That is, the orthogonal section
	$$\Gamma\left(\mathcal{F}\right)^{\perp}=\left\{Y\in \mathfrak{X}(M):\; \omega(X,Y)=0, \;X\in \Gamma\left(\mathcal{F}\right)\right\}$$
	of  $\Gamma\left(\mathcal{F}\right)$
	is equal to $\Gamma\left(\mathcal{F}\right)$. A bi-Lagrangian structure on  $M$ consists on a pair $(\mathcal{F}_{1},\mathcal{F}_{2})$ of transversal Lagrangian foliations together with a symplectic form $\omega$. As a consequence, $TM=T\mathcal{F}_{1}\oplus T\mathcal{F}_{2}$. We denote  by $\mathcal B_l(M)$ the set of   bi-Lagrangian  structures on $M$.
	
	Let $(\mathcal{F}_{1},\mathcal{F}_{2})$ be a bi-Lagrangian structure on a symplectic $2n$-manifold
	$(M,\omega)$. Every point in $M$ has an open neighborhood $U$ which is the domain of a chart whose local coordinates $(p^1,\dots,p^{n},q^1,\dots,q^{n})$ are such that
	\begin{equation*}
	\begin{cases}
	\Gamma(\mathcal{F}_1)_{\mid U}=\left<\frac{\partial}{\partial
		p^1},\dots,\frac{\partial}{\partial p^n}\right>,\vspace{0.25cm}\\
	\Gamma(\mathcal{F}_2)_{\mid U}=\left<\frac{\partial}{\partial
		q^{1}},\dots,\frac{\partial}{\partial q^{n}}\right>.\end{cases}
	\end{equation*}
	Such a chart, and such local coordinates, are said to be adapted to the bi-Lagrangian structure $(\mathcal{F}_{1},\mathcal{F}_{2})$. Moreover, if
	$$\omega=\sum_{i=1}^{n}dq^i\wedge dp^i,$$
	then such a chart, and such local coordinates, are said to be adapted to the bi-Lagrangian structure $(\omega,\mathcal{F}_{1},\mathcal{F}_{2})$.
	
	Let $\pi: M\times \mathbb{R}^{2n}\longrightarrow M$ be the natural projection. We define
	$\Gamma\left(\mathcal{F}_1^{\pi}\right)$ and $\Gamma\left(\mathcal{F}_2^{\pi}\right)$ as follows
	\begin{equation*}
	\begin{cases}
	\Gamma\left(\mathcal{F}_1^{\pi}\right)=\Gamma\left(\mathcal{F}_1\right)+\left<\frac{\partial}{\partial \xi_{n+1}},\dots,\frac{\partial}{\partial \xi_{2n}}\right>\subset \Gamma\left(T\left(M\times \mathbb{R}^{2n}\right)\right),\vspace{0.25cm}\\
	\Gamma\left(\mathcal{F}_2^{\pi}\right)=\Gamma\left(\mathcal{F}_2\right)+\left<\frac{\partial}{\partial \xi_1},\dots,\frac{\partial}{\partial \xi_n}\right>\subset \Gamma\left(T\left(M\times \mathbb{R}^{2n}\right)\right)\end{cases}
	\end{equation*}
	where $\xi_i$, $i=1,\dots, 2n$ are coordinates in $\mathbb{R}^{2n}$.

	Let $(M, \omega,\mathcal{F}_1,\mathcal{F}_2)$ be a bi-Lagrangian  $2n$-manifold. Let us write
	$$\mbox{Lift}( M, \omega,\mathcal{F}_1,\mathcal{F}_2)=(M\times\mathbb{R}^{2n},\tilde{\omega},   \mathcal{F}_1^{\pi},\mathcal{F}_2^{\pi}), $$
	and
	$$\mbox{Lift}^{k+1}( M, \omega,\mathcal{F}_1,\mathcal{F}_2)=\mbox{Lift}^{k}(M\times\mathbb{R}^{2n},\tilde{\omega},   \mathcal{F}_1^{\pi},\mathcal{F}_2^{\pi}), \; k\in\mathbb{N}. $$
	
	We show that, $\mbox{Lift}^{k}( M, \omega,\mathcal{F}_1,\mathcal{F}_2)$ exists for every $k\in\mathbb{N}$ (see Corollary~\ref{Bicor1}); this means,  $(M, \omega,\mathcal{F}_1,\mathcal{F}_2)$ is infinitely  liftable.

	Note that the set $\left(\mathbb{R}^{m}\right)^*$ of  linear forms on $\mathbb{R}^{m}$ and $\mathbb{R}^{m}$ are diffeomorphic. Depending on the context, $\mathbb{R}^{m}$ will sometimes be considered as  $\left(\mathbb{R}^{m}\right)^*$.
	
	A symplectomorphism $\psi$ between two symplectic manifolds $(M_1,\omega_1)$ and $(M_2,\omega_2)$ is a diffeomorphism    $\psi: M_1\longrightarrow M_2$ such that
	$\psi^*\omega_2=\omega_1$. Observe that the set  $Symp(M_1,\omega_1)$ of all symplectomorphisms from $(M_1,\omega_1)$ to itself is a group.
	
	Let  $Conn(M) $ be the set of linear connections on $M$.
	Let $\nabla\in Conn(M) $. The torsion tensor $T_{\nabla}$ (or simply $T$ if there is no ambiguity) and curvature tensor $R_{\nabla}$  (or simply $R$) are given respectively by
	$$ T_{\nabla}(X,Y)=\nabla_XY-\nabla_YX-[X,Y], \; X,Y\in\mathfrak{X}(M)$$
	and
	$$ R_{\nabla}(X,Y)Z=\nabla_X{\nabla_YZ}-\nabla_Y{\nabla_X^Z}-\nabla_{[X,Y]} , \; X,Y,Z\in\mathfrak{X}(M)$$
	where $[X,Y]:=X\circ Y-Y\circ X$ is the Lie bracket of $X$ and $Y$.

	 The set  of all bi-Lagrangian structures on $M$ is denoted by  $\mathcal{B}_{l}(M) $.
	We say that a bi-Lagrangian structure is affine when its Hess connection $\nabla$ is a curvature-free connection; that is, $\nabla$ is flat.  We denote by  $\mathcal{B}_{lp}(M) $ the set of  affine  bi-Lagrangian  structures on $M$. The set $\mathcal{B}_{lp}(M) $ is characterized  in Theorem~\ref{c10}.
	
	We say that a connection $\nabla$
	\begin{enumerate}
		\item[-] parallelizes $\omega$ if  $\nabla\omega=0$; this means,
		\begin{equation} \label{Bieq3}
		\omega(\nabla_{X}{ Y},Z)+\omega(Y,\nabla_{ X}{Z})=X\omega(Y,Z),\; X,Y, Z\in\mathfrak{X}(M);
		\end{equation}
		\item[-] preserves $\mathcal{F}$ if  $\nabla {\Gamma\left(\mathcal{F}\right)}\subseteq \Gamma\left(\mathcal{F}\right)$; more precisely,
		\begin{equation}  \label{Bieq4}
		\nabla_XY\in\Gamma\left(\mathcal{F}\right),\; (X,Y)\in \mathfrak{X}(N)\times\Gamma\left(\mathcal{F}\right).
		\end{equation}	
	\end{enumerate}

	Let $f,g\in C^{\infty}(M)$. The Poisson bracket $\{f,g\}$ of  $f$ and $g$ is the smooth function defined by
	$$\{f,g\}:=\omega (X_f,X_g)$$
	where $X_f$ is the unique vector field verifying $\omega(X_f,Y)=-df(Y)$ for all $Y\in\mathfrak{X}(M)$. We call $X_f$ the Hamiltonian vector field with Hamiltonian function $f$.
	
	Einstein summation convention: an index repeated as sub and superscript in a product represents summation over the range of the index. For example,
	$$\lambda^j\xi_j=\sum_{j=1}^n \lambda^j\xi_j.$$
	In the same way,
	$$X^j\frac{\partial}{\partial y^j}=\sum_{j=1}^nX^j\frac{\partial}{\partial y^j}.$$

	Let $k\in\mathbb{N}$.  Instead of $\{1,2,\dots,k\}$ we will simply write $[k]$. The expression
	$I_k$ stands for the $k\times k$ identity matrix in $\mathbb{R}$.




	\subsection{Technical tools}
	In this part, we present results that we will need in the following.
	\subsubsection{Symplectic manifolds}
	These manifolds provide ideal spaces for some dynamics. The collection of all symplectic manifolds forms a category where arrow or morphism set between two objects (symplectic manifolds) is the set of symplectomorphism between them. Among many results on this category, the cotangent bundle of a manifold is endowed with a so-called tautological 2-form, and a diffeomorphism between two manifolds lifts as a symplectomorphism on their cotangent bundles endowed with their respective tautological  2-forms. This part is devoted to the precise formulations of these results. For more familiarization with the concepts covered in this section, the reader is referred to \cite{dasilva, paul}.

	Let $M$ be a $m$-manifold and let $q:T^*M\longrightarrow M$ be the natural projection.
	The tautological 1-form or Liouville 1-form $\theta$ is defined by
	$$ \theta_{(x,\alpha_x)}(v)=\alpha_x\left(T_{(x,\alpha_x)}q(v)\right),\;  (x,\alpha_x)\in T^{*}M,\,v\in T_xM, $$
	and its exterior differential $d\theta$  is called  the canonical symplectic form or Liouville 2-form  on  the cotangent
	bundle $T^{*}M$.	
	
	Note that for any coordinate chart $(U, x^1,\dots,x^m)$ on $M$, with associated cotangent coordinate chart $(T^*U, x^1,\dots,x^m, \xi_1,\dots,\xi_m)$ we have
	$$\theta=\sum_{1}^{m} \xi_idx_i,$$
	and
	$$d\theta=\sum_{1}^{m} d\xi_i\wedge dx_i.$$
	\begin{proposition}\label{liouville}
		Let $M$ be a manifold. The cotangent bundle $T^{*}M$ of $M$ endowed with   the canonical symplectic form $d\theta$ is a symplectic manifold. 	
	\end{proposition}

	\begin{proposition}\label{lifting of symplecto} Let $M_1$ and $M_2$ be two diffeomorphic smooth manifolds, and  let $\varphi:M_1\longrightarrow M_2$  be a
		diffeomorphism. The lift
		\begin{equation*} \hat{\varphi}:
		z=(x,\alpha_x)\longmapsto(\varphi(x),
		(\varphi^{-1*}\alpha)_{\varphi(x)})\end{equation*} of $\varphi$ is
		a symplectomorphism from $ (T^{*}M_1,d\theta_1)$ to
		$(T^{*}M_2,d\theta_2)$ where $d\theta_1$ and $d\theta_2$ are  the canonical symplectic forms on $T^{*}M_1$ and $T^{*}M_2$ respectively.
	\end{proposition}
	\begin{remark}\label{liouville, lifting of symplecto}
		Let $M$ be a $m$-manifold.
		Proposition~\ref{liouville} and 	Proposition~\ref{lifting of symplecto} still hold by replacing $T^{*}M$ with $M\times\mathbb{R}^{m}$.	
	\end{remark}
	\subsubsection{Bi-Lagrangian (Para-K\"{a}hler) manifolds}
	These manifolds has been intensively explored in the past years, see \cite{2, 11, 3, GB1, FR1, FR2, 7, FE}. Among the many reasons to study them, they are the areas of geometric quantization and of Koszul-Vinberg Cohomology. In this part, we briefly give some needed results concerning  Hess connections and affine bi-Lagrangian structures.
	
	The Hess or bi-Lagrangian connection of a bi-Lagrangian structure is defined from the following theorem.

	\begin{theorem}\cite[Theor. 1]{7}\label{b20}
		Let	$(M,\omega,\mathcal{F}_1,\mathcal{F}_2)$ be a bi-Lagrangian manifold. There exists  a unique torsion-free connection $\nabla$ on $M$ such $\nabla$ parallelizes $\omega$ and preserves both foliations.
	\end{theorem}
	
	Bi-Lagrangian connections  are explicitly defined in the following result, see  \cite[p. 14]{2},  \cite[p. 360]{11}, \cite[p. 65]{3}.
	\begin{proposition}\label{c13} Let $(M,\omega,\mathcal{F}_1,\mathcal{F}_2)$ be a bi-Lagrangian manifold. The Hess connection $\nabla$ of $(\omega,\mathcal{F}_1,\mathcal{F}_2)$ is
		\begin{equation} \nabla_{(X_1, X_2)}{(Y_1,Y_2)} = (D(X_1,Y_1
		)+[X_2,Y_1]_1 , D(X_2,Y_2)+[X_1,Y_2]_2)  \label{17c}\end{equation}
		where $D:\,\mathfrak{X}(M)\times\mathfrak{X}(M)\longmapsto \mathfrak{X}(M)$ is the map verifying
		\begin{equation} i_{D(X,Y )}\omega= L_Xi_Y\omega,
		\label{18c}\end{equation} 	
		and 	$X_i$ is the  $\mathcal{F}_i$-component of $X$ for each $i\in[2]$.
	\end{proposition}
	
	The following result characterizes affine bi-Lagrangian structures.
	\begin{theorem}\cite[Theor. 2]{7} \label{c10}
		Let  $(\omega,\mathcal{F}_1,\mathcal{F}_2)$ be a bi-Lagrangian structure on a       $2n$-manifold $M$ with
		$\nabla$ as its Hess connection. Then the following assertions are equivalent.
		\begin{description}
			\item[a)] The connection $\nabla$ is flat.
			\item[b)] Each point of $M$ has a coordinate chart adapted to $(\omega,\mathcal{F}_1,\mathcal{F}_2)$.
		\end{description}
	\end{theorem}
	



	\section{Statements and proofs of results}

	\subsection{Statements of  results}
	Our first result presents lifted bi-Lagrangian structures on the trivial bundles of some manifolds.
	\begin{theorem}\label{Bitheo3}
		Let $M$ be a $2n$-manifold endowed with  a bi-Lagrangian structure  $(\omega,\mathcal{F}_1,\mathcal{F}_2)$. Then
		$(\tilde{\omega},   \mathcal{F}_1^{\pi},\mathcal{F}_2^{\pi})$ is a bi-Lagrangian structure  on
		$M\times\mathbb{R}^{2n}$ where $\tilde{\omega}=\pi^{*}\omega+d\theta$.\end{theorem}
	
	\begin{corollary}\label{Bicor1}	A  bi-Lagrangian  $2n$-manifold $(M, \omega,\mathcal{F}_1,\mathcal{F}_2)$ is infinitely liftable; that is,
		$\mbox{Lift}^{n}( M, \omega,\mathcal{F}_1,\mathcal{F}_2) $
		exists for every $n\in\mathbb{N}$.	
	\end{corollary}

	Before continuing to state our results, it is necessary to precise the following.
	\begin{remark}\label{act}Let $(M,\omega)$ be a symplectic manifold and let $\psi,\varphi:M\longrightarrow M$ be two diffeomorphisms.
		
		Observe that
		\begin{enumerate}
			\item[-] $\left(\psi\circ\varphi\right)_*=	\psi_*\circ\varphi_*$;
			\item[-]the map  $\nabla^\psi:  \mathfrak{X}(M)\times\mathfrak{X}(M)\longrightarrow
			\mathfrak{X}(M)$, $(X,Y)\longmapsto
			\psi_{\ast}\nabla_{\psi_{\ast
					X}^{-1}}{\psi_{\ast Y}^{-1}}$ is an element of $Conn(M)$ for all $\nabla\in Conn(M),$
			and
			\begin{align*}
			\nabla^{\psi\circ\varphi}&=\left(\psi\circ\varphi\right)_*\nabla_{\left(\psi\circ\varphi\right)_*^{-1}}{\left(\psi\circ\varphi\right)_*^{-1}}
			\\	&=	\psi_*\circ\varphi_*\nabla_{\varphi_*^{-1} \circ \psi_*^{-1}}\varphi_*^{-1} \circ \psi_*^{-1} \\
			&=\left(\nabla^{\varphi}\right)^{\psi}.
			\end{align*}
		\end{enumerate}
		Therefore, the symplectomorphism group  $Symp(M,\omega)$ of $(M,\omega)$ acts on the left of
		\begin{enumerate}
			\item[-]
			$\mathfrak{X}(M)$   as follows  \begin{align*}
			Symp(M,\omega)\times\mathfrak{X}(M)&\longrightarrow
			\mathfrak{X}(M)\\
			(\psi,X) &\longmapsto \psi_{\ast} X
			\label{Be1}	
			\end{align*}
			\item[-]   $Conn(M)$ as follows
			\begin{align*}
			Symp(M,\omega)\times Conn(M) &\longrightarrow  Conn(M)\\
			(\psi,\nabla) &\longmapsto \nabla^\psi
			\end{align*}

		\end{enumerate}
	\end{remark}
	In the following result, we show how a bi-Lagrangian structure can be pushed forward by a diffeomorphism. It will play a fundamental role in the proofs of Theorem~\ref{Bitheo2} and Proposition~\ref{Biprop1}.
	\begin{lemma}\label{Bilem1}
		Let $(M,\omega,\mathcal{F}_1,\mathcal{F}_2)$ be a bi-Lagrangian manifold  with $\nabla$ as its Hess connection, and let $N$ be a manifold which is diffeomorphic to $M$. Then for any diffeomorphism $\psi : M\longrightarrow N$,  $(( \psi^{-1})^*\omega,\psi_*\mathcal{F}_1, \psi_*\mathcal{F}_2)$ is a bi-Lagrangian structure on $N$, with $\nabla^{\psi}$ as its Hess connection.
		Moreover, if $(\omega,\mathcal{F}_1,\mathcal{F}_2)$ is affine, then so is $(( \psi^{-1})^*\omega,\psi_*\mathcal{F}_1, \psi_*\mathcal{F}_2)$.
	\end{lemma}
	\begin{theorem}\label{Bitheo2}
		Let $M$ be a manifold  endowed with a bi-Lagrangian structure. 	
		The map
		\begin{align*}
		\triangleright:{ Symp(M,\omega)}\times{\mathcal B_l(M)}
		&\longrightarrow {\mathcal B_l(M)}\\
		(\psi,(\mathcal{F}_1,\mathcal{F}_2)) &\longmapsto  (\psi_\ast\mathcal{F}_1,\psi_\ast
		\mathcal{F}_2)
		\end{align*}
		is a left group action.	Moreover, for every $(\psi,(\mathcal{F}_1,\mathcal{F}_2)) \in { Symp(M,\omega)}\times{\mathcal B_l(M)}$, the
		Hess connection of  $(\psi_\ast\mathcal{F}_1,\psi_\ast
		\mathcal{F}_2)$ is $\nabla^\psi$   where $\nabla$ is that of  $(\mathcal{F}_1,\mathcal{F}_2)$, and the inclusion
		$\triangleright({ Symp(M,\omega)}\times{\mathcal B_{lp}(M)})\subset \mathcal B_{lp}(M)$ holds.
	\end{theorem}

	We claim that if $M\in\mathcal{M}^{\pi}$ is endowed with a bi-Lagrangian structure, then  every  bi-Lagrangian structure on $M$ can be lifted on $M^{\pi}=TM$ or $T^*M$.
	The precise formulation of the claim is as follows:
	\begin{proposition}\label{Biprop1}
		Let $M$ be a $2n$-manifold endowed with a bi-Lagrangian structure  $(\omega,\mathcal{F}_1,\mathcal{F}_2)$, and belonging to $\mathcal{M}^{\pi}$. There exists
		a diffeomorphism	$\Psi : M\times\mathbb{R}^{2n}\longrightarrow M^{\pi}$ such that
		$(( \Psi^{-1})^*\tilde{\omega},\Psi_*\mathcal{F}_1^{\pi}, \Psi_*\mathcal{F}_2^{\pi})$ is a bi-Lagrangian structure on $M^{\pi}=TM$ or $T^*M$.
	\end{proposition}
	
	Observe that every symplectomorphism on $(M,\omega)$ can be lifted as a symplectomorphism on  $(M\times\mathbb{R}^{2n},\tilde{\omega})$, this follows directly from Proposition~\ref{lifting of symplecto} and Remark~\ref{liouville, lifting of symplecto}. Moreover, every bi-Lagrangian structure $(\omega,\mathcal{F}_1,\mathcal{F}_2)$ on $M$ can be lifted as a  bi-Lagrangian structure $(\tilde{\omega},   \mathcal{F}_1^{\pi},\mathcal{F}_2^{\pi})$ on $M\times\mathbb{R}^{2n}$ (Theorem~\ref{Bitheo3}).
	It is therefore natural to ask: How does $\triangleright$ lift on $M\times\mathbb{R}^{2n}$? What is the relationship between  $\hat{\triangleright}$ the lifting of $\triangleright$  and  $\tilde{\triangleright}$ the action (in the sense of Theorem~\ref{Bitheo2}) of $Symp(M\times\mathbb{R}^{2n},\tilde{\omega})$  on $\mathcal{B}_{l}(M\times\mathbb{R}^{2n})?$
	By combing Theorem~\ref{Bitheo3} and Theorem~\ref{Bitheo2} we have the following result which is a powerful answer to the first question, and rightful the title of this paper.
	
	\begin{corollary}\label{cor2}
		The action $\triangleright$ can be lifted infinitely in a similar sense as in Corollary~\ref{Bicor1}.\end{corollary}
	
	The set $\hat{Symp}(M,\omega)$ is defined as follows:
	$$\hat{Symp}(M,\omega)=\{\hat{\psi}\in Symp(M\times\mathbb{R}^{2n},\tilde{\omega}),\;\psi\in Symp(M,\omega)\},$$ see Proposition~\ref{lifting of symplecto}  and Remark~\ref{liouville, lifting of symplecto} for more details.
	
	\begin{proposition}\label{Biprop2} Let $(\omega,\mathcal{F}_1,\mathcal{F}_2)$ be a bi-Lagrangian structure  on a manifold $M$, and let $\psi$ be a symplectomorphism on $(M,\omega).$ If each point of $M$ has a coordinate chart
		$(U, p^1,\dots,p^{n},q^1,\dots,q^{n})$ which is adapted to $(\mathcal{F}_1,\mathcal{F}_2)$ and verifying
		\begin{equation}\hat{\psi}_*\frac{\partial}{\partial p^i}\in\Gamma((\psi_{*}\mathcal{F}_1)^{\pi}),\; i\in[n], \end{equation}
		then\begin{equation*}\hat{\psi}\hat{\triangleright}(\mathcal{F}_1^{\pi},\mathcal{F}_2^{\pi})=\hat{\psi}\tilde{\triangleright}(\mathcal{F}_1^{\pi},\mathcal{F}_2^{\pi}):=(\hat{\psi}_*\mathcal{F}_1^{\pi},\hat{\psi}_*\mathcal{F}_2^{\pi}).\end{equation*}
	\end{proposition}
	
	\begin{remark} Theorem~\ref{Bitheo3}, Theorem~\ref{Bitheo2} and  Proposition~\ref{Biprop2} which constitute the main results of this paper
		can be summarized as follows:
		Let $(M,\omega,\mathcal{F}_1,\mathcal{F}_2)$ be a bi-Lagrangian manifold.
		The  diagram\vspace{0.25cm}\\
		\begin{center}
			\setlength{\unitlength}{1mm} \thicklines
			\begin{picture}(40,20)
			\put(-9,0){$(\mathcal{F}_1,\mathcal{F}_2)$}
			\put(36.5,0){$(\psi_{*}\mathcal{F}_1,\psi_{*}\mathcal{F}_2)$}
			\put(-10,20){$(\mathcal{F}_1^{\pi},\mathcal{F}_2^{\pi})$}
			\put(31,20){$((\psi_{*}\mathcal{F}_1)^{\pi},(\psi_{*}\mathcal{F}_2)^{\pi})$}
			\put(18,-2){$\psi_{*}$}
			\put(16,23){$\hat{\psi}_{*}$}
			\put(-7.5,4.5){\rotatebox{90}{$\pi$:\:Lift}}
			\put(6,4.25){$\triangleright $:\:Push forward}
			\put(-4,3.5){ \vector(0,1){15.5}}
			\put(48,3.5){\vector(0,1){15.5}}
			\put(5,21.55){	\vector(1,0){24}}
			\put(6,1.55){\vector(1,0){30}}
			\end{picture}
		\end{center}\smallskip
		existssr. Moreover, the above diagram is commutative for some $(\mathcal{F}_1,\mathcal{F}_2)$.
	\end{remark}

	


	\subsection{Proofs of results}
	We start this section with the following observation.
	\begin{remark}\label{Birem1} Let $(M,\omega,\mathcal{F}_1,\mathcal{F}_2)$ be a bi-Lagrangian manifold with a curvature-free Hess connection $\nabla$. Let   $(G,F)$ be the para-K\"{a}hler structure associated to $(\omega,\mathcal{F}_1,\mathcal{F}_2)$. Each point of $M$ has a coordinate chart  $(U, p^1,\dots,p^{n},q^1,\dots,q^{n})$ such that for all $x\in U$
		$$ \omega_{x}=\left(\begin{array}{cc}
		0&I_n\\
		-I_n & 0
		\end{array}\right),\;  F_{x}=\left(\begin{array}{cc}
		I_n&0\\
		0 & -I_n
		\end{array}\right), \mbox{ and } G_{x}=\left(\begin{array}{cc}
		O&I_n\\
		I_n & 0
		\end{array}\right) .    $$
		
		Let $x\in U$. Since $R_{\nabla} = 0$, then by Theorem~\ref{c10} there exists a  coordinate chart  $(U, p^1,\dots,p^{n},q^1,\dots,q^{n})$ such that
		\begin{equation}\label{B1}\Gamma\left(\mathcal{F}_{1}\right)=\left<\frac{\partial}{\partial
			p^1}, \dots, \frac{\partial}{\partial
			p^n}\right>,\;  \Gamma\left(\mathcal{F}_{1}\right)=\left<\frac{\partial}{\partial
			q^1}, \dots, \frac{\partial}{\partial
			q^n}\right> \end{equation}
		and
		$$ \omega_{x}=\left(\begin{array}{cc}
		0&I_n\\
		-I_n & 0
		\end{array}\right).  $$
		By  (\ref{B1}), we have
		$$F_{x}=\left(\begin{array}{cc}
		I_n&0\\
		0 & -I_n
		\end{array}\right).  $$
		Thus since $G_x(X_x,Y_x) =\omega_x(F_x(X_x),Y_x)$, we obtain
		$$ G_{x}=\left(\begin{array}{cc}
		0&I_n\\
		I_n & 0
		\end{array}\right) . $$
	\end{remark}

	\subsubsection{Lifted bi-Lagrangian structures}
	
	\begin{proof}[Proof of Theorem~\ref{Bitheo3}]
		
		Let $(M,\omega,\mathcal{F}_1,\mathcal{F}_2)$ be a bi-Lagrangian  $2n$-manifold.
		We are going to show that $(M\times\mathbb{R}^{2n},\tilde{\omega},\mathcal{F}_1^{\pi},\mathcal{F}_2^{\pi})$ is a bi-Lagrangian manifold.

		Let
		$(U, p^1,\dots,p^{n},q^1,\dots,q^{n})$ be a  coordinate chart adapted to $(\mathcal{F}_1,\mathcal{F}_2)$, with $(U\times \mathbb{R}^{2n}, p^1,\dots,p^{n},q^1,\dots,q^{n}, \xi_1,\dots,\xi_{2n})$ as its associated bundle coordinate chart. Then
		\begin{equation}\label{b9}
		\begin{cases}
		\Gamma(\mathcal{F}_1^{\pi})=\left<\frac{\partial}{\partial
			p^1},\dots,\frac{\partial}{\partial   p^n},\frac{\partial}{\partial
			\xi_{n+1}},\dots,\frac{\partial}{\partial \xi_{2n}}\right>,\vspace{0.25cm}\\
		\Gamma(\mathcal{F}_2^{\pi})=\left<\frac{\partial}{\partial
			q^{1}},\dots,\frac{\partial}{\partial q^{n}},\frac{\partial}{\partial
			\xi_1},\dots,\frac{\partial}{\partial
			\xi_n}\right>,\end{cases}
		\end{equation}
		and  the canonical symplectic form	$d\theta$ is defined on $U\times \mathbb{R}^{2n}$  as follows $$d\theta=\sum\limits_{1}^{2n} d\xi_i\wedge dx_i.$$

		Observe that $\tilde{\omega}=\pi^*\omega+d\theta$ is antisymmetric (as sum of two antisymmetric forms), closed (pull-backs  commute with exterior derivatives) and non-degenerate (direct).  That is, $\tilde{\omega}$ is a symplectic form on $M\times\mathbb{R}^{2n}$.
		
		By (\ref{b9}), it follows that $(\mathcal{F}_1^{\pi},\mathcal{F}_2^{\pi})$ is a transversal pair of smooth Lagrangian distributions on $(M\times\mathbb{R}^{2n}, \tilde{\omega})$. Thus, it remains to show that $\mathcal{F}_i^{\pi}, i=1,2$ are completely integrable. Since the distributions $\mathcal{F}_1^{\pi}$ and $\mathcal{F}_2^{\pi}$ are similar, we only treat the case $\mathcal{F}_1^{\pi}$.
		
		We are going to show that
		\begin{equation}\label{integra of lift foliation}
		d\theta([X,Y],Z)=0,\; X,Y,Z\in\Gamma(\mathcal{F}_1^{\pi}).
		\end{equation}
		Note that
		\begin{equation*}
		d\theta([X,Y],Z)=[X,Y]\theta(Z)-Z\theta([X,Y])-\theta([[X,Y],Z]).\end{equation*}
		Let us write
		\begin{equation*}
		\begin{cases}
		(y^i)_{i=1,\dots,2n}=((p^i)_{i=1,\dots,n},(\xi_{i})_{i=n+1,\dots,2n}),\\
		X=X^i\frac{\partial}{\partial y^i},Y=Y^j\frac{\partial}{\partial
			y^j}\;\mbox{ and }\;Z=Z^k\frac{\partial}{\partial
			y^k}.\end{cases}\end{equation*} Then \begin{equation*}
		\begin{cases}
		[X,Y]=\mu^j\frac{\partial}{\partial y^j},\\
		[[X,Y],Z]=\lambda^j\frac{\partial}{\partial
			y^j},\end{cases}\end{equation*} where
		\begin{equation*}
		\begin{cases}
		\mu^j=X^i\frac{\partial Y^j}{\partial y^i}-Y^i\frac{\partial
			X^j}{\partial y^i},\\
		\lambda^j=\mu^i\frac{\partial Z^j}{\partial y^i}-Z^i\frac{\partial
			\mu^j}{\partial y^i}.\end{cases}\end{equation*} Thus,
		\begin{align}
		[X,Y]\theta(z)&=\mu^i\frac{\partial}{\partial
			y^i}(Z^k\xi_k),\tag{$e_1$}\label{Bieq5}\\
		\theta([[X,Y],Z])&=\lambda^j\xi_j, \tag{$e_2$}\label{Bieq6}
		\\Z\theta([X,Y])&=\frac{\partial}{\partial y^k}(\mu^i\xi_i). \tag{$e_3$}\label{Bieq7}
		\end{align}
		Therefore
		\begin{equation*}
		d\theta([X,Y],Z)=\mbox{(\ref{Bieq5})}-\mbox{(\ref{Bieq6})}-\mbox{(\ref{Bieq7})}=0.
		\end{equation*}
		
		Observe that for every $X\in \mathfrak{X}\left(M\times\mathbb{R}^{2n}\right)$, $\pi_*X$  depends only on components of $X$ on  $M$. Thus,  by (\ref{integra of lift foliation})
		\begin{equation*}
		\tilde{\omega}([X,Y],Z)=0,\; X,Y,Z\in\Gamma(\mathcal{F}_1^{\pi}).
		\end{equation*}
		This completes the proof that $(\tilde{\omega},   \mathcal{F}_1^{\pi},\mathcal{F}_2^{\pi})$ is a bi-Lagrangian structure on $M\times \mathbb{R}^{2n}$.

		 This completes the proof of Theorem~\ref{Bitheo3}.	
	\end{proof}
	
	Corollary~\ref{Bicor1} follows by combining system (\ref{b9})
	and Theorem~\ref{c10}.

	\subsubsection{Action of symplectomorphism group}
	
	\begin{lemma}\label{Bilem3} Let  $(M,\omega)$ be a symplectic manifold endowed with a  Lagrangian foliation $\mathcal F$, and let $N$ be a manifold such that $M$ and $N$ are diffeomorphic. Let
		$\psi: M\longrightarrow N$  be a diffeomorphism.  Then $\psi _\ast\mathcal F$ is a Lagrangian foliation on  $(N,\left(\psi^{-1}\right)^*\omega)$.
	\end{lemma}
	\begin{proof}Let $X=\psi _\ast{X}',Y=\psi _\ast{Y}'\in \Gamma (\psi _\ast\mathcal F) $ where $X',Y'\in \Gamma (\mathcal F)$.
		
		On the one hand,
		\begin{align}\nonumber
		\left(\psi^{-1} \right)^\ast\omega(X,Y)&=\omega(\psi_\ast^{-1} {X},\psi _\ast^{-1} {Y})\circ\psi^{-1}\\
		&=\omega(X',Y')\circ\psi^{-1}.\label{Bieq1}
		\end{align}
		Since $\mathcal F$ is a Lagrangian foliation, by (\ref{Bieq1}) we get 	
		\begin{equation}
		\left(\psi^{-1} \right)^\ast\omega(X,Z)=0,\; X\in \Gamma (\psi _\ast\mathcal F) \Longleftrightarrow \;Z\in \Gamma (\psi _\ast\mathcal F).\label{1c}
		\end{equation}
		That is,
		$\psi _\ast\mathcal F$ is a Lagrangian distribution on $(N,\left(\psi^{-1}\right)^*\omega)$.
		
		On the other hand,
		\begin{equation}
		[X,Y]=[\psi _\ast X',\psi _\ast Y' ]=\psi_\ast[X',Y']. \label{Be2}
		\end{equation}
		By combining  (\ref{Be2}) and (\ref{1c}), we get $[X,Y]\in \Gamma (\psi _\ast\mathcal F)$ for all $ X,Y\in \Gamma (\psi _\ast\mathcal F)$; this means, $\psi _\ast\mathcal F$ is completely integrable, so it is a Lagrangian foliation on $(N,\left(\psi^{-1}\right)^*\omega)$ as claimed.
	\end{proof}
	
	\begin{proof}[Proof of Lemma~\ref{Bilem1}]

		Let $(\omega,\mathcal{F}_1,\mathcal{F}_2)$ be a bi-Lagrangian structure on a manifold $M$, let $N$ be a manifold which is diffeomorphic to $M$, and let $\psi : M\longrightarrow N$ be a diffeomorphism.

		Since  $\psi$ is a diffeomorphism, then $\psi_*$ is bijective. By combining this with Lemma~\ref{Bilem3}, it follows that $(\mathcal{F}_1,\mathcal{F}_2)$ is a bi-Lagrangian foliation on $(N,\left(\psi^{-1}\right)^*\omega)$. Now let  $\nabla $ be the Hess connection of $(\omega,\mathcal{F}_1,\mathcal{F}_2)$, we claim that $\nabla^\psi$ is that of  $(\left(\psi^{-1}\right)^*\omega, \psi_*\mathcal{F}_1,\psi_*\mathcal{F}_2)$; more precisely,
		\begin{enumerate}
			\item  $\nabla^\psi$ is a torsion-free connection:
			\begin{equation*}
			(\nabla^\psi)_{X}{Y} - (\nabla^\psi)_{Y}{X} = [{X},{Y}];
			\end{equation*}
			\item $\nabla^\psi$ parallelizes:
			\begin{equation*}
			\left(\psi^{-1}\right)^*\omega(\left(\nabla^{\psi}\right)_{X}{ Y},Z)+\left(\psi^{-1}\right)^*\omega(Y,\left(\nabla^{\psi}\right)_{ X}{Z})=X\left(\left(\psi^{-1}\right)^*\omega(Y,Z)\right)
			\end{equation*}
			for any $X,Y, Z\in\mathfrak{X}(N)$;
			\item $\nabla^\psi$ preserves both foliations:
			\begin{equation*}
			\left(	\nabla^\psi\right)_XY\in\Gamma\left(\psi_\ast\mathcal{F}_i\right),\;(X,Y)\in \mathfrak{X}(N)\times\Gamma\left(\psi_\ast\mathcal{F}_i\right), i=1,2.
			\end{equation*}	
		\end{enumerate}
		\begin{enumerate}
			\item Let  $X,Y\in\mathfrak{X}(N)$. Since $\nabla$ is a torsion-free connection, we get
			\begin{align*}
			(\nabla^\psi)_{X}{Y} - (\nabla^\psi) _{Y}{X} &=
			\psi_\ast(\nabla_{\psi^{-1}_\ast X}{\psi^{-1}_\ast Y}-\nabla_{\psi^{-1}_\ast Y}{\psi^{-1}_\ast X})\\
			&= \psi_\ast[\psi^{-1}_\ast X,\psi^{-1}_\ast Y]\\
			&= [{X},{Y}].
			\end{align*}
			\item  Let $X,Y, Z\in\mathfrak{X}(N)$. Observe that
			\begin{align}
			\Delta &: =\left(\psi^{-1}\right)^*\omega(\nabla_{X}^{\psi}{ Y},Z)
			+\left(\psi^{-1}\right)^*\omega(Y,\nabla_{ X}{Z})
			\nonumber\\&=\left(\psi^{-1}\right)^*\omega(\psi_\ast(\nabla_{\psi^{-1}_\ast X}{\psi^{-1}_\ast Y}),Z)+(\psi^{-1})^*\omega(Y,\psi_\ast(\nabla_{\psi^{-1}_\ast X}{\psi^{-1}_\ast
				Z}))
			\nonumber\\&=\left[\omega(\nabla_{\psi^{-1}_\ast X}{\psi^{-1}_\ast Y},\psi^{-1}_\ast Z)
			+\omega(\psi^{-1}_\ast Y,\nabla_{\psi^{-1}_\ast X}{\psi^{-1}_\ast
				Z})\right]\circ\psi^{-1}
			\nonumber\\&=\left[\left( \psi^{-1}_\ast X\right)\left( \omega(\psi^{-1}_\ast Y,\psi^{-1}_\ast
			Z)\right)\right]\circ \psi^{-1}
			\label{Bieq8}\\&= \left[\left( \psi^{-1}_\ast X\right)\left( \left(\psi^{-1}\right)^*\omega( Y,
			Z)\circ \psi\right)    \right]\circ \psi^{-1}		
			\nonumber\\&=X\left(\left(\psi^{-1}\right)^*\omega(Y,Z)\right).\nonumber
			\end{align}
			Note that  (\ref{Bieq8}) comes from the fact that $\nabla$ parallelizes $\omega$ (see \ref{Bieq3}).

			\item Let $X\in \mathfrak{X}(M)$ and $Y=\psi_{\ast} Y'\in\Gamma (\psi_\ast\mathcal{F}_i)$. We have
			\begin{equation}\label{Be3}(\nabla^\psi)_XY=\psi_\ast(\nabla_{\psi^{-1}_\ast
				X}{\psi^{-1}_\ast Y})=\psi_\ast(\nabla_{\psi^{-1}_\ast
				X}{Y^\prime}).\end{equation}
			Since  $\nabla$ preserves $\mathcal{F}_i$ (see \ref{Bieq4}), from (\ref{Be3})
			we have \begin{equation*} \psi_\ast(\nabla_{\psi^{-1}_\ast
				X}{Y^\prime})\in \Gamma(\psi_\ast\mathcal{F}_i).\end{equation*}
			That is,
			\begin{equation*} (\nabla^\psi)_XY \in\Gamma(\psi_\ast\mathcal{F}_i).\end{equation*}
		\end{enumerate}
		We use the following observation to end the proof.
		\begin{remark}\label{Birem3}
			
			Observe that
			\begin{equation*}
			T_{\nabla ^{\psi}}(X,Y)=\psi_{\ast}(T_{\nabla}(\psi
			_{\ast}^{-1}X,\psi
			_{\ast}^{-1}Y)),\;X,Y\in\mathfrak{X} (N),\end{equation*}
			and
			\begin{equation*}
			R_{\nabla ^{\psi}}(X,Y)Z=\psi_{\ast}(R_{\nabla}(\psi
			_{\ast}^{-1}X,\psi_{\ast}^{-1}Y)\psi
			_{\ast}^{-1}Z),\; X,Y,Z\in\mathfrak{X} (N).\end{equation*}
		\end{remark}	
		Thus, if $R_{\nabla}= 0 $,  then  $R_{\nabla^{\psi}}= 0 $. As a consequence, 	$\triangleright$ preserves affine bi-Lagrangian structures.	
		
		This completes the proof of Lemma~\ref{Bilem1}.
	\end{proof}
	\begin{proof}[Proof of Theorem~\ref{Bitheo2}]
		By Lemma~\ref{Bilem1},
		$\triangleright$ is well defined, and by Remark~\ref{Birem3}, the inclusion  $\triangleright( Symp(M,\omega)\times\mathcal{B}_{lp}(M))\subset \mathcal{B}_{lp}(M)$ holds.    The action proprieties of $\triangleright$ come from those of the action of $Symp(M,\omega)$ on $\mathfrak{X}(M)$ (Remark~\ref{act}). Theorem~\ref{Bitheo2} is proved.
	\end{proof}
	
	\begin{proof}[Proof of Proposition~\ref{Biprop1}]
		By combining Theorem~\ref{Bitheo3} and Lemma~\ref{Bilem1}, Proposition~\ref{Biprop1} follows.
	\end{proof}

	By equality (\ref{Bieq2}) and Lemma~\ref{Bilem1}, we get the following result.
	\begin{proposition}\label{Biprop3}
		Let $(\omega, \mathcal{F}_1,\mathcal{F}_2)$ be a  bi-Lagrangian structure on a manifold $M$ with  $(G,F)$ as its associated para-K\"{a}hler structure, and let $\psi: M\longrightarrow N$ be a diffeomorphism.  Then the paracomplex structure $F^{\psi}$  associated  to $(\left(\psi^{-1}\right)^*\omega, \psi_\ast\mathcal{F}_1,\psi_\ast\mathcal{F}_2)$  is
		\begin{equation*}
		F^{\psi}(X)=\psi_\ast {F(\psi^{-1}_{*}X)},\; X\in\mathfrak{X} (N).
		\end{equation*}
	\end{proposition}

	\subsubsection{Lifting of  $\triangleright$}
	
	\begin{proof}[Proof of Proposition~\ref{Biprop2}]

		We start by defining $\hat{\triangleright}$ a lift of $\triangleright$, and its action properties will come from those of $\triangleright$.
		\begin{proposition}\label{Biprop4} Let
			$(M,\omega)$ be a symplectic manifold endowed with a bi-Lagrangian structure. Then
			a lift $\hat{\triangleright}$ of
			$\triangleright$
			can be defined by
			\begin{equation*}
			\hat{\psi}\hat{\triangleright}(\mathcal{F}_1^{\pi},\mathcal{F}_2^{\pi})=
			(\psi\triangleright(\mathcal{F}_1,\mathcal{F}_2))^{\pi}=((\psi_{*}\mathcal{F}_1)^{\pi},(\psi_{*}\mathcal{F}_2)^{\pi})
			\end{equation*}
			for all
			$\psi\in Symp(M,\omega)$ and $(\omega,\mathcal{F}_1,\mathcal{F}_2)\in\mathcal{B}_{l}(M)$.
		\end{proposition}

		\begin{proposition}\label{Biprop6}
			Let $\hat{\psi}\in\hat{Symp}(M,\omega)$ and
			$(\omega,\mathcal{F}_1,\mathcal{F}_2)\in\mathcal{B}_{l}(M)$ such that
			\begin{equation}\hat{\psi}_*(\mathcal{F}_1^{\pi})\subseteq
			(\psi_{*}\mathcal{F}_1)^{\pi}.\label{30c}\end{equation} Then
			\begin{equation*}\hat{\psi}\hat{\triangleright}(\mathcal{F}_1^{\pi},\mathcal{F}_2^{\pi})=\hat{\psi}\tilde{\triangleright}(\mathcal{F}_1^{\pi},\mathcal{F}_2^{\pi}).\label{19c}\end{equation*}
		\end{proposition}
		\begin{proof}
			Note that the diagram\\\smallskip
			\begin{center}
				\setlength{\unitlength}{1mm} \thicklines
				\begin{picture}(40,20)
				\put(3,0){$M$}
				\put(40,0){$M$}
				\put(-1.75,20){$M\times\mathbb{R}^{2n}$}
				\put(36,20){$M\times\mathbb{R}^{2n}$}
				\put(20,-2){$\psi$}
				\put(20,23){$\hat{\psi}$}
				\put(2.5,10){$\pi$}
				\put(43.5,10){$\pi$}
				\put(4,19){ \vector(0,-1){15}}
				\put(43,19){\vector(0,-1){15}}
				\put(13.5,21.65){	\vector(1,0){20}}
				\put(7.5,1.55){\vector(1,0){32}}
				\end{picture}
			\end{center}\smallskip
			is commutative. By lifting it on the tangent bundle, we get\\\smallskip

			\begin{center}
				\setlength{\unitlength}{1mm} \thicklines
				\begin{picture}(40,20)
				\put(3,0){$TM$}
				\put(40,0){$TM$}
				\put(-6,20){$T(M\times\mathbb{R}^{2n})$}
				\put(31,20){$T(M\times\mathbb{R}^{2n})$}
				\put(20,-2){$\psi_{*}$}
				\put(20,23){$\hat{\psi}_{*}$}
				\put(1.25,10){$\pi_{*}$}
				\put(43.5,10){$\pi_{*}$}
				\put(4,19){ \vector(0,-1){15}}
				\put(43,19){\vector(0,-1){15}}
				\put(15.5,21.65){	\vector(1,0){14}}
				\put(10.5,1.55){\vector(1,0){29}}
				\end{picture}
			\end{center}\smallskip	
			
			and  by the following decompositions
			\begin{equation*}
			\begin{cases}
			\Gamma(TM)=\Gamma\left(\mathcal{F}_1\right)\oplus\Gamma\left(\mathcal{F}_2\right)=\Gamma(\psi_{*}\mathcal{F}_1)\oplus\Gamma(\psi_{*}\mathcal{F}_2)\\
			\Gamma(T(M\times\mathbb{R}^{2n}))=\Gamma\left(\mathcal{F}_{1}^{\pi}\right)\oplus\Gamma\left(\mathcal{F}_{2}^{\pi}\right)=\Gamma((\psi_{*}\mathcal{F}_1)^{\pi})\oplus\Gamma((\psi_{*}\mathcal{F}_2)^{\pi})
			\end{cases}
			\end{equation*}
			we obtain\\\smallskip

			\begin{center}
				\setlength{\unitlength}{1mm} \thicklines
				\begin{picture}(40,20)
				\put(-20,0){$\Gamma\left(\mathcal{F}_1\right)\oplus\Gamma\left(\mathcal{F}_2\right)$}
				\put(30,0){$\Gamma(\psi_{*}\mathcal{F}_1)\oplus\Gamma(\psi_{*}\mathcal{F}_2)$}
				\put(-20,20){$\Gamma\left(\mathcal{F}_1^{\pi}\right)\oplus\Gamma\left(\mathcal{F}_2^{\pi}\right)$}
				\put(25,20){$\Gamma((\psi_{*}\mathcal{F}_1)^{\pi})\oplus\Gamma((\psi_{*}\mathcal{F}_2)^{\pi})$}
				\put(15,-2){$\psi_{*}$}
				\put(13,23){$\hat{\psi}_{*}$}
				\put(-10,10){$\pi_{*}$}
				\put(49,10){$\pi_{*}$}
				\put(-7,19){ \vector(0,-1){15.5}}
				\put(48,19){\vector(0,-1){15.5}}
				\put(8,21.65){	\vector(1,0){14}}
				\put(8.5,1.55){\vector(1,0){21}}
				\end{picture}
			\end{center}\smallskip

			Thus, since $\hat{\psi}_*$ is bijective, then  by (\ref{30c}) we obtain

			\begin{equation*}\hat{\psi}_*(\mathcal{F}_1^{\pi})= (\psi_{*}\mathcal{F}_1)^{\pi}\;\mbox{ and }\;\hat{\psi}_*(\mathcal{F}_2^{\pi})=
			(\psi_{*}\mathcal{F}_2)^{\pi}.\end{equation*}
		\end{proof}
		
		In the next result, we give a condition to obtain    \textbf(\ref{30c}). We use the previous notations.

		\begin{proposition}\label{prop1}
			Let $\psi$ be a symplectomorphism on $(M,\omega).$ If each point of $M$ has a coordinate chart
			$(U, p^1,\dots,p^{n},q^1,\dots,q^{n})$ adapted to $(\mathcal{F}_1, \mathcal{F}_2)$ such that
			\begin{equation}\hat{\psi}_*\frac{\partial}{\partial p^i}\in\Gamma((\psi_{*}\mathcal{F}_1)^{\pi}),\; i\in[n],\label{eq1}\end{equation}
			then\begin{equation*}\hat{\psi}_*(\mathcal{F}_1^{\pi})\subseteq
			(\psi_{*}\mathcal{F}_1)^{\pi}.\label{25c}\end{equation*}
		\end{proposition}
		\begin{proof}
			Let
			$(U\times\mathbb{R}^{2n}p^1,\dots,p^{n},q^1,\dots,q^{n},\xi_1\dots\xi_{2n})$ be
			a bundle coordinate chart associated to  $(U, p^1,\dots,p^{n},q^1,\dots,q^{n})$. We have
			\begin{equation*}
			\Gamma(\hat{\psi}_*(\mathcal{F}_1^{\pi}))=\left<\hat{\psi}_*\frac{\partial}{\partial
				p^1},\dots\hat{\psi}_*\frac{\partial}{\partial
				p^n},\hat{\psi}_*\frac{\partial}{\partial
				\xi_{n+1}},\dots,\hat{\psi}_*\frac{\partial}{\partial
				\xi_{2n}}\right>.\end{equation*} Thus, by (\ref{eq1}) it remains to  show that
			\begin{equation*}\hat{\psi}_*\frac{\partial}{\partial \xi_i}\in
			(\psi_{*}\mathcal{F}_1)^{\pi},\; i=n+1,\dots,2n.\end{equation*}
			Let $i=1,\dots,n,\;j=n+1,\dots,2n$. We have
			\begin{equation*}
			\tilde\omega\left(\hat{\psi}_*\frac{\partial}{\partial
				p^i},\hat{\psi}_*\frac{\partial}{\partial
				\xi_j}\right)=\tilde\omega\left(\frac{\partial}{\partial
				p^i},\frac{\partial}{\partial \xi_j}\right)\circ\hat{\psi}^{-1}=0.\end{equation*}
			Then   $
			\hat{\psi}_*\frac{\partial}{\partial\xi_i} $ belongs to
			$\Gamma(((\psi_{*}\mathcal{F}_1)^{\pi})^{\bot})$ which is equal to
			$\Gamma((\psi_{*}\mathcal{F}_1)^{\pi})$.
			
			This completes the proof of Proposition~\ref{prop1}.
		\end{proof}
		By combining Proposition~\ref{prop1} and Proposition~\ref{Biprop6}, Proposition~\ref{Biprop2} follows.
	\end{proof}
	



	
	\section{ Examples on  $(\mathbb{R} ^2, \omega)$}
	We start this part by introducing Christoffel symbols. Let $G$ be a pseudo-Riemannian metric in $\mathbb{R}^2$ defined as follows:
	$G(\partial_i,\partial_j)=G_{ij}$ where $\partial_1=\dfrac{\partial}{\partial x}$ and
	$\partial_2=\dfrac{\partial}{\partial y}$. Let $\nabla$ be the Levi-Civita connection of $G$. The Christoffel symbols $\Gamma_{ij}^k$; $i,j,k=1,2$ of $\nabla$ are defined as follows:
	$\nabla_{\partial_i}{\partial_j}=\Gamma_{ij}^k\partial_k.$
	More precisely,
	$$\Gamma_{ij}^k=\dfrac{1}{2}G^{kl}\left(\partial_jG_{il}+\partial_iG_{lj}-\partial_lG_{ij}\right).$$
	
	Our first examples are described on  affine bi-Lagrangian structures. Suppose that $( \omega,\mathcal{F}_1,\mathcal{F}_2)$ is an affine bi-Lagrangian structure on $\mathbb{R}^2$. By Remark~\ref{Birem1} there exists a  system coordinate  $(x,y)$ such that
	$$ \omega=dy\wedge dx,\;  F= \dfrac{\partial}{\partial x}dx -\dfrac{\partial}{\partial y} dy\mbox{ and } G= dx\otimes dy   $$
	where $(G,F)$ is the associated para-Kh\"{a}ler structure of $( \omega,\mathcal{F}_1,\mathcal{F}_2)$.
	As a consequence, the Hess connection  associated to $( \omega,\mathcal{F}_1,\mathcal{F}_2)$ (which is the Levi-Civita connection of $G$, see \cite{FR1, FR2, 1, FE}) is  trivial; that is,  its Christoffel symbols vanish. That is  why  we will present a second example with a non-trivial Hess connection.
	\begin{remark}\label{Birem2}
		Let $M$ be a manifold. Every $X\in\mathfrak{X}(M)$ without singularity (this means, $X_z\neq 0$ for every $z\in M$) generates (induces) a foliation $\mathcal{F}^X$ on $M$. In particular, if $M$ is a 2-manifold endowed with a symplectic form $\omega$, then $\mathcal{F}^X$ is Lagrangian independently of the symplectic form $\omega$. As a consequence, any $X, Y\in\mathfrak{X}(M)$ such that $Dim\,\left<X_z, Y_z\right>=2$ for every point $z\in M$ generates a bi-Lagrangian structure on $(M, \omega)$, independently of the  symplectic form $\omega$.
	\end{remark}
	\subsection{Case of $(\mathbb{R} ^2, \omega=dy\wedge dx)$}
	\subsubsection{Action of  $Symp(\mathbb{R} ^2, \omega)$ on $\mathcal B_l(\mathbb{R} ^2)$}
	\paragraph{Symplectomorphism group on $(\mathbb{R} ^2, \omega)$}
	
	$$Symp(\mathbb{R} ^2, \omega):=\{ \psi\in Diff(\mathbb{R}^2):\,\det T_x\psi=1\} $$
	where
	$$ \det T_x\psi:=\frac{\partial \psi _1}{\partial
		x^1}\frac{\partial \psi _2}{\partial x^2}-\frac{\partial \psi
		_2}{\partial x^1}\frac{\partial \psi _1}{\partial x^2}.
	$$
	
	For technical reasons, we describe our example on the subgroup
	$Symp_a(\mathbb{R} ^2, \omega)$ of $Symp(\mathbb{R} ^2, \omega)$ defined by:
	$$Symp_a(\mathbb{R} ^2, \omega)=\left\{\psi_{AB}:(x,y)\mapsto A\left(\begin{array}{c}
	x\\y\end{array}\right)+B,\; A\in SL_2(\mathbb{R}), B\in \mathbb{R}^2\right\}$$
	where
	$SL_2(\mathbb{R})=\{A\in M_2(\mathbb{R}):\; \det A=1\} $.

	The action of  $ Symp_a(\mathbb{R} ^2, \omega)$ on $\mathfrak{X}(\mathbb{R}
	^2)$ is:
	$Symp_a(\mathbb{R} ^2, \omega)\times \mathfrak{X}(\mathbb{R} ^2)  \longrightarrow
	\mathfrak{X}(\mathbb{R} ^2)$, $(\psi,X) \longmapsto  \psi _{\ast}X.$	
	More precisely,  for any
	\begin{equation*} (x,y)\in \mathbb{R} ^2,\; \psi_{\ast(x,y)}=\left( \begin{array}{lll}\smallskip \alpha & \beta\\
	\gamma & \delta \end{array}\right)\; \mbox{ and }\; X=\left(\begin{array}{c}
	X^1\\X^2\end{array}\right),
	\end{equation*}
	
	\begin{equation*} \psi_{\ast (x,y)}X_{(x,y)}=\left( \begin{array}{lll}\smallskip \alpha & \beta\\
	\gamma & \delta \end{array}\right) \left(\begin{array}{lll}
	X^{1}(x,y) \\ X^{2}(x,y)\end{array} \right)=\left(\begin{array}{lll}
	\alpha X^{1}(x,y)+\beta X^{2}(x,y) \\ 	\gamma X^{1}(x,y)+\delta X^{2}(x,y)\end{array} \right).\end{equation*}
	Let $(\mathcal{F}^x,\mathcal{F}^y)$	be the pair of two decompositions of $\mathbb{R} ^2$ constituted of all horizontal and vertical lines respectively.  That is, $$\mathcal{F}^x=\left\{\mathcal{F}^x_b=\mathbb{R}\times\{b\}\right\}_{b\in\mathbb{R}} \mbox{ and } \mathcal{F}^y=\left\{\mathcal{F}^y_a=\{a\}\times\mathbb{R}\right\}_{a\in\mathbb{R}}.$$  As a consequence,
	$$\Gamma\left(\mathcal{F}^x\right)= \mathbb{R}\times\{0\}=\left<\dfrac{\partial}{\partial x} \right> \mbox{ and } \Gamma\left(\mathcal{F}^y\right)= \{0\}\times\mathbb{R}=\left<\dfrac{\partial}{\partial y} \right>.$$
	By Remark~\ref{Birem2},	$(\mathcal{F}^x,\mathcal{F}^y)$ is a bi-Lagrangian structure on $(\mathbb{R} ^2, \omega)$. Before graphing it, we describe the action (in the sense of Theorem~\ref{Bitheo2}) of $ Symp_a(\mathbb{R} ^2, \omega)$ on $(\mathcal{F}^x,\mathcal{F}^y)$.
	
	Let $\psi\in Symp_a(\mathbb{R} ^2, \omega)$.
	Observe that 	
	\begin{equation*}
	\begin{cases}
	\psi_{\ast}\mathcal{F}^x_b:y=\frac{\gamma}{\alpha}x+\frac{b}{\alpha},\; b\in \mathbb{R},\vspace{0.25cm} \\
	\Gamma\left(\psi
	_{\ast}\mathcal{F}^x\right)= \left<\alpha\dfrac{\partial}{\partial x}+\gamma\dfrac{\partial}{\partial y} \right>,
	\end{cases}
	\mbox{ and }
	\begin{cases}
	\psi_{\ast}\mathcal{F}^y_a: y=\frac{\delta}{\beta}x-\frac{a}{\beta},\; a\in \mathbb{R},\vspace{0.25cm}\\\Gamma\left(\psi
	_{\ast}\mathcal{F}^y\right)= \left<\beta\dfrac{\partial}{\partial x}+\delta\dfrac{\partial}{\partial y} \right>.
	\end{cases} 			
	\end{equation*}	
	The  para-complex structure  $F^\psi$  associated to
	$(\psi_{\ast}\mathcal{F}^x, \psi_{\ast}\mathcal{F}^y)$ is	
	\begin{equation*}
	F^{\psi}(\psi _{\ast}\frac{\partial}{\partial
		x})=\alpha \dfrac{\partial}{\partial x}+\gamma\dfrac{\partial}{\partial y}	\mbox{ and }
	F^{\psi}(\psi
	_{\ast}\frac{\partial}{\partial y})=-\beta\dfrac{\partial}{\partial x}-\delta\dfrac{\partial}{\partial y}.
	\end{equation*}
	
	Similar results are obtained for another  bi-Lagrangian structure belonging in  the orbit  $$\mathcal{B}_{0}=\left\{ (\psi_{\ast}\mathcal{F}^x, \psi_{\ast}\mathcal{F}^y),\; \psi\in Symp_a(\mathbb{R} ^2, \omega)\right\}$$ of  $(\mathcal{F}^x,\mathcal{F}^y)$  with respect to
	$\triangleright| Symp_a(\mathbb{R} ^2, \omega)\times\mathcal B_l(\mathbb{R}
	^2).$
	
	The bi-Lagrangian structure $(\mathcal{F}^x,\mathcal{F}^y)$ can be represented as follows.
	
	\begin{figure}[htp]	
		\centering
		\setlength{\unitlength}{9cm}
		\begin{picture}(1, 1)
		\multiput(0.04,0)(0.04,0){24}
		{\line(0,1){1}}

		\multiput(0,0.04)(0,0.04){24}
		{\line(1,0){1}}
		
		\end{picture}
		\caption{The bi-Lagrangian structure $(\mathcal{F}^x,\mathcal{F}^y)$} 
		\label{a map 1}    
	\end{figure}

	Now, we are going to apply Proposition~\ref{Biprop6} to   $\triangleright| Symp_a(\mathbb{R} ^2, \omega)\times\mathcal{B}_{0}$
	
	Note that for any \begin{equation*}
	A=\left( \begin{array}{lll}\smallskip \alpha & \beta\\
	\gamma & \delta \end{array}\right)\in SL_2(\mathbb{R})\;\mbox{ and }\;B=\left(\begin{array}{lll} a\\
 	b\end{array}\right)\in\mathbb{R}^2,\end{equation*}
	the map
	\begin{equation*}\psi_{AB}: \,\left(\begin{array}{lll} x \\
	y\end{array}\right)\in \mathbb{R}^2 \longmapsto A\left(\begin{array}{lll} x \\
	y\end{array}\right)+B \end{equation*}
	is invertible with the explicit inverse   \begin{equation}\label{21c}
	\psi_{AB}^{-1}: \,\left(\begin{array}{lll} x \\
	y\end{array}\right)\in \mathbb{R}^2 \longmapsto A^{-1}\left(\begin{array}{lll} x \\
	y\end{array}\right)-A^{-1}B  \end{equation}
	
	\paragraph{Lifting of affine symplectomorphism}
	\begin{proposition}\label{affi sympl}
		An affine  symplectomorphism on $(\mathbb{R}^2,\omega)$ lifts  as an affine symplectomorphism on $(\mathbb{R}^4, \tilde{\omega})$. That is,  $ \hat{Symp_a}(\mathbb{R} ^2)\subset Symp_a(\mathbb{R} ^4, \tilde{\omega})$.
	\end{proposition}
	
	\begin{proof}
		Let  $\psi\in Symp_a(\mathbb{R} ^2, \omega)$. We have
		\begin{equation*}\hat{\psi}:z=(p,\xi_p)\longmapsto(\psi(p),(\psi^{-1*}\xi)_{\psi(p)}.
		\end{equation*}
		Let $(x,y,s,t)$ be a coordinate system on $\mathbb{R}^4$. Then  $z=(x,y,s,t)$, $\xi=sdx+tdy$ and $\tilde{\omega}=dy\wedge dx +ds\wedge dx+ dt\wedge dy .$ Moreover, since \begin{equation*}\psi(x,y)=(\alpha x+\beta y+a,\gamma x+\delta y+b)\end{equation*}
		for some $\alpha,\beta,\gamma,\delta, a,b\in\mathbb{R}$
		verifying  $\alpha\delta-\beta\gamma=1,$  then by (\ref{21c}), we get
		\begin{equation*}\psi^{-1}(x,y)=(\delta x-\beta y+\delta a-\beta b,-\gamma x+\alpha
		y-\delta a+\alpha b).\end{equation*}
		As a consequence,
		\begin{equation*}
		(\psi^{-1*}\xi)_{\psi(p)}=(s(p)\delta-t(p)\gamma)dx+(\alpha t(p)-\beta s(p))dy.
		\end{equation*}
		Then
		\begin{equation*}
		\hat{\psi}(z)=(\alpha x+\beta y+a,\gamma x+\delta
		y+b,s\delta-t\gamma,-\beta s+\alpha t).
		\end{equation*}
		Therefore
		\begin{equation}
		T_z\hat{\psi}=\hat{\psi}_*z=\left( \begin{array}{llrr}\smallskip \alpha & \beta & 0 & 0\\
		\gamma & \delta & 0 & 0\\
		0 &  0 & \delta & -\beta\\
		0 &  0 &-\gamma &\alpha\end{array}\right)=\left( \begin{array}{ll}\smallskip A & 0\\
		0 & A^{-1}\end{array}\right)\label{c31}\end{equation}
		where \begin{equation*}
		A=\left( \begin{array}{lll}\smallskip \alpha & \beta\\
		\gamma & \delta \end{array}\right).\end{equation*}
		This ends the proof of Proposition~\ref{affi sympl}. \end{proof}

	{Lifting of  $(\mathcal{F}^x,\mathcal{F}^y)$.}
	
	Note that
	\begin{equation*}\Gamma(\mathcal{F}^x)=\left<\frac{\partial}{\partial
		x}\right>\mbox{ and }\Gamma(\mathcal{F}^y)=\left<\frac{\partial}{\partial
		y}\right>.\label{23c}\end{equation*} Thus,
	\begin{equation*}\Gamma((\mathcal{F}^x)^{\pi})=\left<\frac{\partial}{\partial
		x},\frac{\partial}{\partial
		t}\right>\;\mbox{ and }\;\Gamma((\mathcal{F}^y)^{\pi})=\left<\frac{\partial}{\partial
		y},\frac{\partial}{\partial s}\right>. \label{22c}\end{equation*}
	\begin{proposition}\label{27c}
		Let
		$\psi\in Symp_a(\mathbb{R} ^2, \omega)$. Then $\hat{\psi}_*((\mathcal{F}^y)^{\pi})\subseteq
		(\psi_*\mathcal{F}^y)^{\pi}.$
	\end{proposition}
	\begin{proof}
		Let $\psi\in Symp_a(\mathbb{R} ^2, \omega)$. By (\ref{c31}) we get
		\begin{align*}\hat{\psi}_*\frac{\partial}{\partial
			x}=\psi_*\frac{\partial}{\partial
			x}\in \Gamma((\psi_*\mathcal{F}^x)^{\pi}).\end{align*}
		And by Proposition~\ref{Biprop2} we have the result.
	\end{proof}
	
	{Lifting of $\mathcal{B}_0$.}
	
	We are going to explicit  $((\psi_*\mathcal{F}^x)^{\pi}, (\psi_*\mathcal{F}^y)^{\pi})$
	for some $\psi\in Symp_a(\mathbb{R} ^2, \omega)$.
	
	Let $\psi\in Symp_a(\mathbb{R} ^2, \omega)$,  by  Proposition~\ref{27c}
	we get
	\begin{equation*}\hat{\psi}_*((\mathcal{F}^x)^{\pi})\subseteq
	(\psi_*\mathcal{F}^x)^{\pi}. \end{equation*}Thus, by Proposition~\ref{Biprop6} we obtain
	\begin{equation*}((\psi_*\mathcal{F}^x)^{\pi},(\psi_*\mathcal{F}^y)^{\pi})=\hat{\psi}_*((\mathcal{F}^x)^{\pi}, (\mathcal{F}^y)^{\pi}),\end{equation*}
	and  Proposition~\ref{affi sympl} implies that \begin{equation*}
	\hat{\psi}_*=\left( \begin{array}{ll}\smallskip A & 0\\
	0 & A^{-1}\end{array}\right);\label{c15}\end{equation*}
	where \begin{equation*}
	A=\left( \begin{array}{lll}\smallskip \alpha & \beta\\
	\gamma & \delta \end{array}\right).\end{equation*}Therefore
	\begin{equation*}\begin{cases}\Gamma((\psi_*\mathcal{F}^y)^{\pi})=\left<\hat{\psi}_*\frac{\partial}{\partial x},\hat{\psi}_*\frac{\partial}{\partial t}\right>\vspace{0.25cm}\\
	\Gamma((\psi_*\mathcal{F}^x)^{\pi})=\left<\hat{\psi}_*\frac{\partial}{\partial
		y},\hat{\psi}_*\frac{\partial}{\partial s}\right>.
	\end{cases}\end{equation*}

	\subsection{A bi-Lagrangian structure on $(\mathbb{R}^2,\omega=hdy\wedge dx)$  }
	In this part, we present  $(\mathcal{P},\mathcal{F}^y)$ the bi-Lagrangian structure on $(\mathbb{R}^2,\omega)$ constituted of parabolas and vertical lines, and calculate its  Hess connection.
	
	\subsubsection{Description de $(\mathcal{P},\mathcal{F}^y)$}
	The foliation $\mathcal{P}$ is described as follows:
	$$\mathcal{P}=\left\{\mathcal{P}_{(a,b)}:y=x^2+b-a^2\right\}_{(a,b)\in\mathbb{R}^2}. $$
	Thus,
	\begin{equation*}
	\begin{cases}
	\Gamma(\mathcal{P})=\left<\frac{\partial}{\partial
		x}+2x\frac{\partial}{\partial y}\right>,\vspace{0.25cm}\\
	\Gamma(\mathcal{F}^x)=\left<\frac{\partial}{\partial y}\right>.
	\end{cases}
	\end{equation*}
	Let us write
	\begin{equation*}
	\begin{cases}
	U=\frac{\partial}{\partial x}+2x\frac{\partial}{\partial y},\\
	V=\frac{\partial}{\partial y}.
	\end{cases}
	\end{equation*}

	By Remark~\ref{Birem2},	$(\mathcal{P},\mathcal{F}^y)$ is a bi-Lagrangian structure on $(\mathbb{R}^2,\omega)$. It can be represented as follows.
	
	\begin{figure}[htp]	
		\centering
		\setlength{\unitlength}{9cm}
		\begin{picture}(1, 1)
		\multiput(0,0)(0.04,0){26}
		{\line(0,1){1}}

		\multiput(0,-0.75)(0,0.05){16}
		{\qbezier(-0.01,1)(0.5,0.50)(1.01,1)}
		\end{picture}
		\caption{ The bi-Lagrangian structure 	$(\mathcal{P},\mathcal{F}^y)$}  
		\label{a map 2}    
	\end{figure}
	
	\subsubsection{The Hess connection  of $(\mathcal{P},\mathcal{F}^y)$}

	We are going to determine
	\begin{equation*}
	\nabla_{(U,0)}{(U,0)},\;\nabla_{(0,V)}{(0,V)},\;\nabla_{(U,0)}{(0,V)}\mbox{ and }\nabla_{(0,V)}{(U,0)}.
	\end{equation*}
	By (\ref{17c}) it is enough to calculate
	\begin{equation*}
	D(U,U),\; D(V,V),\; D(U,0),\; D(0,V).
	\end{equation*}
	Let us write $x^1=x$ and $x^2=y$.
	
	Let  $ X,Y,Z\in\mathfrak{X}(\mathbb{R}^2)$. From
	(\ref{18c}), we get
	\begin{equation*}
	\omega(D(X,Y),Z)=X\omega(Y,Z)-\omega(Y,[X,Z]).
	\end{equation*}
	Then
	\begin{align*}
	\omega(D(X,Y),Z)&=X[h(dx^2(Y)dx^1(Z)-dx^2(Z)dx^1(Y))]\\
	& -h(dx^2(Y)dx^1([X,Z])-dx^2([X,Z])dx^1(Y)).
	\end{align*}
	Thus, on the one hand,
	\begin{equation*}
	\omega\left(D(U,U),\frac{\partial}{\partial
		x^j}\right)=U[h(2x\delta_{1j}-\delta_{2j})]-2h\delta_{1j}.
	\end{equation*}
	On the other hand,
	\begin{equation*}
	\omega\left(D(U,U),\frac{\partial}{\partial
		x^j}\right)=h[\delta_{1j}dx^2(D(U,U)-\delta_{2j}dx^1(D(U,U)].
	\end{equation*}
	Then\begin{equation*}
	\begin{cases}
	hdx^1(D(U,U))=U(h),	\\
	hdx^2(D(U,U))=U(2xh)-2h.\end{cases}
	\end{equation*}
	Therefore,
	\begin{equation*}
	D(U,U)=\frac{U(h)}{h}U.
	\end{equation*}
	In the same way as before,
	\begin{equation*}
	D(V,V)=\frac{V(h)}{h}V.
	\end{equation*}
	Moreover, since  $[\frac{\partial}{\partial x^i},\frac{\partial}{\partial x^j}]=0$, then
	\begin{eqnarray*}
		[U,V]=\left[\frac{\partial}{\partial
			x^1}+2x^1\frac{\partial}{\partial x^2}, \frac{\partial}{\partial x^2} \right]=0.
	\end{eqnarray*}
	Thus,
	\begin{equation*}
	\begin{cases}
	\nabla_{(U,0)}{(U,0)}=\left(\frac{U(h)}{h},0\right),\vspace{0.25cm}\\
	\nabla_{(0,V)}{(0,V)}=\left(0,\frac{V(h)}{h}\right),\vspace{0.25cm}\\
	\nabla_{(U,0)}{(0,V)}=\nabla_{(0,V)}{(U,0)}=(0,0).
	\end{cases}
	\end{equation*}
	Therefore   \begin{equation}
	\begin{cases}   \Gamma_{11}^{1}=\frac{U(h)}{h},\\
	\Gamma_{22}^{2}=\frac{V(h)}{h},\\
	\Gamma_{22}^{1}=\Gamma_{12}^{1}=\Gamma_{21}^{1}=0,\\
	\Gamma_{12}^{2}=\Gamma_{21}^{2}=\Gamma_{11}^{2}=0.\end{cases}\label{20c}
	\end{equation}
	\subsubsection{The curvature tensor of  $\nabla$}
	Note that
	\begin{equation*} R(U_i,U_j,)U_k=R_{ijk}^lU_l, \; i,j,k\in[2]
	\end{equation*}
	where $U_1=U$, $U_2=V$ and
	\begin{equation*}
	R_{ijk}^l
	=U_i(\Gamma_{jk}^{l})-U_j(\Gamma_{ik}^{l})+\Gamma_{jk}^{s}\Gamma_{is}^{l}-\Gamma_{ik}^{s}\Gamma_{js}^{l}.
	\end{equation*}
	Thus by  (\ref{20c}),  we get
	\begin{equation}
	\begin{cases}
	R_{211}^1=-R_{121}^1=V(\Gamma_{11}^{1}),\\
	R_{122}^2=-R_{212}^2=U(\Gamma_{22}^{2}),\\
	\mbox{the other coefficients are zero}.
	\end{cases}\label{21cl}
	\end{equation}
	\begin{remark}	By combining Theorem~\ref{c10} and system (\ref{21cl}), $(\omega, \mathcal{P},\mathcal{F}^y)$ is an affine bi-Lagrangian structure on $\mathbb{R}^2$ when $V(\Gamma_{11}^{1})=U(\Gamma_{22}^{2})=0$; in particular, when $h$ is a constant  map.
	\end{remark}

\end{document}